\documentclass[12pt,a4paper]{article}

\usepackage{amsmath}

\usepackage{amssymb}
\usepackage{amsthm}
\usepackage{mathrsfs}
\usepackage{wasysym}
\usepackage{bbm}		
\usepackage{mathtools}
\usepackage[shortlabels]{enumitem}
\usepackage{braket}		
\usepackage{mdwlist}		
\usepackage{xcolor}
\usepackage{tikz}
\usepackage{graphicx}
\usepackage[a4paper, left=2.7cm, right=2.7cm, top=3.5cm, bottom=2cm]{geometry}
\usepackage[pdftex=true,hyperindex=true,pdfborder={0 0 0}]{hyperref}
\usetikzlibrary{patterns,positioning,arrows,decorations.markings,calc,decorations.pathmorphing,decorations.pathreplacing}
\usepackage{cleveref}

\theoremstyle{plain}
\newtheorem{theorem}{Theorem}[section]
\newtheorem{lemma}[theorem]{Lemma}

\newtheorem{proposition}[theorem]{Proposition}
\newtheorem{definition}[theorem]{Definition}

\newtheorem{question}[theorem]{Question}

\theoremstyle{definition}
\newtheorem{remark}[theorem]{Remark}
\newtheorem{example}[theorem]{Example}

\newcommand{\infinito}[1]{{}^\infty #1}

\newcommand{\ZZ}{\mathbb{Z}}			
\newcommand{\NN}{\mathbb{N}}			

\usepackage[all]{xy}
\usepackage{color}
\usepackage{microtype}

\newcommand{\e}{\varepsilon}
\renewcommand{\d}{\delta}

\newcommand{\pacman}[1]{\tikz[baseline=.1em,scale=.4]{
\draw (-.1,0) -- (-.1,.85) -- (.8,.85) -- (.8,0) -- cycle; 
  \draw [fill=#1] (.45,.425) -- (.7,.575) arc (+25:+335:.375) -- cycle;
  \fill (0.55,0.6) circle (.35mm)
}}

\newcommand{\cinco}[1]{\tikz[baseline=.0001em,scale=.28]{
  \draw (-.35,-.2) -- (-.35,1.05) -- (1.05,1.05) -- (1.05,-0.2) -- cycle;
  \draw (-.1,-.2) -- (-.1,1.05);
  \draw (.85,-.2) -- (.85,1.05); 
  \draw [fill=#1] (0.05,0.05) -- (.05,.5) arc (+180:0:.3) -- (.65,0.05) --
  (.55,.2) -- (.45,0.05) -- (.35,.2) -- (.25,0.05) -- (.15,.2) -- cycle;
    \coordinate (eye) at (360*rand:.03);
    \foreach \x in {.17,.43}{
      \fill[white] (\x,.5) circle[radius=.1];
      \fill[black] (\x,.5) ++(eye) circle[radius=.05];
    }
}}

\newcommand{\uno}{\tikz[baseline=.1em,scale=.4]{ 
\draw (-.1,0) -- (-.1,.85) -- (.8,.85) -- (.8,0) -- cycle; 
\draw (.1,0) -- (.1,.85);
\draw (.6,0) -- (.6,.85);
}}

\newcommand{\cero}{\tikz[baseline=.1em,scale=.4]{ 
\draw (-.1,0) -- (-.1,.85) -- (.8,.85) -- (.8,0) -- cycle; 
}}

\newcommand{\punto}[1]{\tikz[baseline=.1em,scale=.4]{
\draw (-.1,0) -- (-.1,.85) -- (.8,.85) -- (.8,0) -- cycle; 
\draw [fill=#1] (0.4,0.42) circle (3mm)
}}

\newcommand{\ghost}[1]{\tikz[baseline=.1em,scale=.4]{
  \draw (-.1,0) -- (-.1,.85) -- (.8,.85) -- (.8,0) -- cycle; 
  \draw [fill=#1] (0.05,0.05) -- (.05,.5) arc (+180:0:.3) -- (.65,0.05) --
  (.55,.15) -- (.45,0.05) -- (.35,.15) -- (.25,0.05) -- (.15,.15) -- cycle;
    \coordinate (eye) at (360*rand:.03);
    \foreach \x in {.17,.43}{
      \fill[white] (\x,.5) circle[radius=.1];
      \fill[black] (\x,.5) ++(eye) circle[radius=.05];
    }
   \draw [fill=white] (0.12,0.2) -- (0.12,0.3) -- (0.2,0.35) -- (0.28,0.3) -- (0.36,0.35) -- (0.44,0.3) -- (0.52,0.35) -- (0.6,0.3)-- (0.6,0.2) -- (0.52,0.25) -- (0.44,0.2) -- (0.36,0.25) -- (0.28,0.2) -- (0.2,0.25) -- cycle;
}}

\newcommand{\key}[1]{\tikz[baseline=.1em,scale=.4]{
  \draw (-.1,0) -- (-.1,.85) -- (.8,.85) -- (.8,0) -- cycle; 
  \draw [fill=#1] (0.05,0.05) -- (.05,.5) arc (+180:0:.3) -- (.65,0.05) --
  (.55,.15) -- (.45,0.05) -- (.35,.15) -- (.25,0.05) -- (.15,.15) -- cycle;
    \coordinate (eye) at (360*rand:.03);
    \foreach \x in {.17,.43}{
      \fill[white] (\x,.5) circle[radius=.1];
      \fill[black] (\x,.5) ++(eye) circle[radius=.05];
    }
   \draw [fill=white] (0.12,0.2) -- (0.12,0.3) -- (0.2,0.35) -- (0.28,0.3) -- (0.36,0.35) -- (0.44,0.3) -- (0.52,0.35) -- (0.6,0.3)-- (0.6,0.2) -- (0.52,0.25) -- (0.44,0.2) -- (0.36,0.25) -- (0.28,0.2) -- (0.2,0.25) -- cycle;
   \draw [fill=green] (0.1,0.55) arc (+270:0:0.15)  -- (0.45,0.7) -- (0.45,0.8) -- (0.55,0.8) -- (0.55,0.7) -- (0.65,0.7) -- (0.65,0.8) -- (0.75,0.8) -- (0.75,0.55) -- cycle;
}}

\newcommand{\cherry}{\tikz[baseline=.1em,scale=.4]{
  \draw (-.1,0) -- (-.1,.85) -- (.8,.85) -- (.8,0) -- cycle; 
  \draw ( .35, .8 ) -- ( .35, .7 ) -- ( .2, .3 );
  \draw (.35,.7) -- (.5, .3);
  \fill[red] (.15,.3) circle[radius=.2];
  \fill[red] (.55,.3) circle[radius=.2];
  
}}

\newcommand{\banana}{\tikz[baseline=.1em,scale=.4]{
\draw (-.1,-.1) -- (-.1,.75) -- (.8,.75) -- (.8,-.1) -- cycle; 
\draw [fill=black] (.0,.6) -- (.55,.6) -- (.55,.65) -- (.0,.65) -- cycle;
\draw [fill=yellow] (.2,.3) -- (.17,.45) -- (.05,.6) arc (+90:-90:.3) -- (.17,.15) -- (.2,.3) (.2,.6) arc (+90:-90:.3) (.35,.6) arc (+90:-90:.3)  (.5,.6) arc (+90:-90:.3);
  
}}

\title{Mean equicontinuity and mean sensitivity on cellular automata}
\author{Luguis de los Santos Ba\~nos and Felipe Garc\'{\i}a-Ramos}

\date{}	

\begin{document}

	\maketitle
	
		\begin{abstract}

		   We show that a cellular automaton (or shift-endomorphism) on a transitive subshift is either almost equicontinuous or sensitive. On the other hand, we construct a cellular automaton on a full shift (hence a transitive subshift) that is neither almost mean equicontinuous nor mean sensitive.  
		\end{abstract}
	
	\section{Introduction}
Sensitivity to initial conditions (or simply sensitivity) is one of the classical notions of chaos on dynamical systems. It was introduced for topological dynamical systems by Guckenheimer \cite{guckenheimer}. By a \emph{topological dynamical system (shortly TDS)} we mean a pair $(X,T)$ such that $X$ is a compact metric space (with metric $d$) and $T:X\rightarrow X$ is continuous. 	A TDS is \emph{sensitive} if there exists $\varepsilon>0$ such that for every non-empty open set $U\subseteq X$ there exist $x,y\in U$ and $n>0$ such that $d(T^nx,T^ny)>\varepsilon$. A notion of order that contrasts sensitivity is equicontinuity (or Lyapunov stablity); a TDS is equicontinuous if $\{ T^n \}_{n\in \NN}$ is an equicontinuous family. Using sensitivity and equicontinuity one can classify transitive topological dynamical systems (see Definition \ref{def: transitivity}). Akin, Auslander and Berg proved that any transitive TDS is either sensitive or almost equicontinuous \cite{akinauslander} (a generalization of the Auslander-Yorke dichotomy \cite{auslander1980}). 
Nonetheless, this classification has some limitations, because sensitivity is not a very strong form of chaos (for example; every non-finite subshift is sensitive; for cellular automata, equicontinuity is strongly connected to local periodicity \cite{garcia2016limit}). Inspired by the notion of mean equicontinuity (or mean Lyapunov stablility) first studied by Fomin \cite{fomin} and Oxtoby \cite{oxtoby}, the notion of mean sensitivity was introduced \cite{lituye,weakforms}. A TDS is \emph{mean sensitive} if there exists $\varepsilon>0$ such that for every non-empty open set $U$ there exist $x,y\in U$ such that $d(T^nx,T^ny)>\varepsilon$ for any $n$ in a set with density bigger than $\varepsilon$. The key difference is ``how many" $n$'s satisfy the condition. For example, the Sturmian subshift is sensitive but not mean sensitive \cite{weakforms}. Similar to the classic case, one can classify transitive TDSs using the mean notions, that is, a transitive TDS is either mean sensitive or almost mean equicontinuous \cite{lituye,weakforms}. Mean equicontinuity/sensitivity has been studied in other recent papers (for instance \cite{downarowiczglasner,garciajagerye,huang2018bounded,garcia2019dynamical,fuhrmann2018structure} or the survey \cite{lisurveymean}) and it is very related to (measurable) discrete spectrum, properties of the maximal equicontinuous factor and quasicrystals. 

Cellular automata (CA) are dynamical systems defined on full-shifts $A^\ZZ$ (or more generally on subshifts). They have been used to model phenomena that are based on local rules in physics, biology and computer science. The notion of sensitivity in CA has been studied in many papers (for example, \cite{kurka1997languages,martin2007damage,gilman1987classes,blanchardtisseur,sablik2008directional,garcia2016limit}). In particular, Kurka proved that any CA (not necessarily transitive) is either sensitive or almost equicontinuous \cite{kuurka2003topological}. Actually one of the main ingredients of this proof is that the full-shift is transitive (with respect to the shift map). Hence, this statement can be generalized to any shift endomorphism on a transitive subshift (see Proposition \ref{dicsub}). So it is natural to ask if, just like in the transitive topological dynamics case, a similar dichotomy to Kurka's holds for the mean versions on cellular automata (on transitive subshifts). 

In this paper we provide the first examples of the study of mean equicontinuity/sensitivity on CA. Firstly, we construct an almost mean equicontinuous CA that is not almost equicontinuous (Theorem \ref{teorema1}). Secondly, we construct a CA that is neither mean sensitive nor almost mean equicontinuous (Theorem \ref{teorema1}). So Kurka's dichotomy does not hold for the mean notions on cellular automata. In conclusion, cellular automata can be divided in the following four disjoint non-empty classes (Theorem \ref{teorema1} and Theorem \ref{teorema 2}): almost equicontinuous, almost mean equicontinuous but not almost equicontinuous, neither almost mean equicontinuous nor mean sensitive, and mean sensitive.

\textbf{Acknowledgements:} 
The authors would like to thank Rafael Alcaraz Barrera for valuable comments. 
The first author receives support from a CONACyT PhD fellowship, and the second author from the CONACyT Ciencia B\'asica project 287764.

\section{Definitions and preliminaries}

\begin{definition}

Let $S\subseteq \ZZ_{\geq 0}$. We define the \textbf{upper density} of $S$  by
$$\overline{D}(S)=\limsup_{n\rightarrow \infty}\frac{\sharp (S\cap \{ 0,\ldots n-1  \} ) }{n}.$$

\end{definition}

A \textbf{topological dynamical system (TDS)} is a pair $(X,T)$ where $X$ is a compact metric space (with metric $d$) and $T:X\rightarrow X$ is continuous.

Transitivity is a topological form of ergodicity. 
\begin{definition}
\label{def: transitivity}
Let $(X,T)$ be a TDS. We say that $(X,T)$ is \textbf{transitive} if for every pair of non-empty open sets $U$ and $V$ there exists $n>0$ such that $T^{-n}U\cap V\neq \emptyset$.

\end{definition}

\begin{definition}
\begin{enumerate}
\item Given a finite non-singular set $A$ (called an alphabet), we define the \textbf{$A$-full shift} as $A^{\ZZ }$. If $X$ is the $A$-full shift for some finite $A$ we say that $X$ is a \textbf{full shift}. 

\item Given $x\in A^{\ZZ}$, we represent the $i$-th coordinate of $x$ as $x_{i}$. Also, given $i,j\in \ZZ$ with $i<j$, we define the finite word $x_{[i,j]}=x_{i}\dots x_{j}$.

\item We endow any full shift with the metric 
\begin{displaymath}
\begin{array}{rcl}
d(x,y) & = & \left\{ \begin{array}{ccl} 2^{-i} & \text{if} \ x\neq y & \text{where} \ i=\min \{ |j| : x_{j}\neq y_{j}  \} ; \\  
0 & \text{otherwise.} &
\end{array} \right.
\end{array}
\end{displaymath}

This metric generates the same topology as the product topology.
\item For any full shift $A^{\ZZ}$, we define the shift map $\sigma:A^{\ZZ}\rightarrow A^{\ZZ}$ by $\sigma (x)_{i}=x_{i+1}$. The shift map is continuous (with respect to the previously defined metric).

\item We say $X$ is a \textbf{subshift} (or shift space) if $X \subseteq A^{\ZZ}$ is closed and $\sigma$-invariant.

\end{enumerate}
\end{definition}

Typically, cellular automata are defined on a full shift. We give a more general definition. These systems are also known as \emph{shift-endomorphisms} or \emph{sliding block-codes}.  

\begin{definition}\label{defCAShift}
We say that $(X,T)$ is a \textbf{cellular automaton (CA)} if $X$ is a subshift and $T:X\rightarrow X$ is continuous and commutes with $\sigma$, i.e., $\sigma \circ T=T\circ \sigma$. 
\end{definition}

As we mentioned in the introduction, cellular automata can be described using local rules. Note that $Tx_i$ represents the $i$th coordinate of the point $Tx$. 

\begin{theorem}[Curtis-Hedlund-Lyndon]\label{CA}
Let $X$ be a subshift and $T:X\rightarrow X$ a function. Then, $(X,T)$ is a  cellular automaton if and only if there exist integers $m\leq a$ and a (local) function $f:A^{a-m+1}\rightarrow A$ such that for any $x\in X$ and any $i\in \ZZ$
$$Tx_{i}=f(x_{[i+m,i+a]}).$$
\end{theorem}
 
\subsection{Sensitivity, equicontinuity and dichotomies}

A subset of a topological space is \textbf{residual (or comeagre)} if it is the intersection of a countable number of dense open sets. 
\begin{definition}
Let $(X,T)$ be a TDS and $x\in X$.  

\begin{enumerate}
\item The point $x$ is an \textbf{equicontinuity point} if $$\forall
\e > 0, \ \exists \d>0 \ \text{such that} \ \ \forall y\in B_{\d }(x), \ \forall n\geq 0, 
\ d(T^{n}x,T^{n}y)<\e. $$ The set of equicontinuity points of $(X,T)$ is denoted by $EQ$.

\item  $(X,T)$ is \textbf{equicontinuous} if 
$EQ=X$.

\item  $(X,T)$ is \textbf{almost equicontinuous} if $EQ$
is a residual set.

\item $(X,T)$ is \textbf{sensitive} if there exists $\e >0$ such that for every non-empty open set $U\subseteq X$ there exist $x,y\in U$ and $n\neq 0$ 
      such that $$d(T^{n}x,T^{n}y)> \e.$$
\end{enumerate}
\end{definition}

Sensitivity and almost equicontinuity can be used to classify transitive topological dynamical systems. 
\begin{theorem}[\cite{akinauslander}]
Transitive topological dynamical systems are sensitive if and only if they are not almost equicontinuous. 
\end{theorem}

A CA satisfies the same dichotomy without assuming transitivity. This result is proved in \cite{kuurka2003topological} for CA on the full shift. Using the same technique we prove the result for CA on transitive subshifts. 

\begin{proposition}
	\label{dicsub}
	Let $(X,\sigma)$ be a transitive subshift and $(X,T)$ a CA. Then, $(X,T)$ is almost equicontinuous if and only if is not sensitive.
\end{proposition}

\begin{proof}
	$\Rightarrow$: Assume that $(X,T)$ is almost equicontinuous. This means that for every open subset $U\subseteq X$, there exists $x\in U$ such that for all $\e >0$ there exists $\d>0$ such that if $d(x,y)<\d$, then we have that $d(T^{n}x,T^{n}y)<\e$ for all $n\geq 0$.
	
	Let $\e>0$. Observe that (using $\frac{\e}{2}$) there exists $\d >0$ such that for all $y,z\in B_{\d}(x)$ and all $n\geq 0$, we have that 
	\begin{displaymath}
	\begin{array}{rcl}
	d(T^{n}y,T^{n}z)&\leq & d(T^{n}y,T^{n}x) + d(T^{n}x,T^{n}z)\\
	& < & \frac{\e}{2} +\frac{\e}{2} =\e.
	
	\end{array}
	\end{displaymath}
	Therefore, $(X,T)$ is not sensitive.

	$\Leftarrow :$ Assume that $(X,T)$ is not sensitive, that is, for all $\e>0$ there exists an open set $U\subseteq X$ such that for all $x,y\in U$ and for all $n\geq 0$, we have that $d(T^{n}x,T^{n}y)<\e$. Now, since T is uniformly continuous, for $\e=1$, there exists $r\geq 0$ such that if $d(x,y)=2^{-r}$, then $d(Tx,Ty)<1$. This implies that for all $x,y\in X$ such that $x_{[-r,r]}=y_{[-r,r]}$, we have that $Tx_{0}=Ty_{0}$. Hence, for all $m\geq 0$, there exist $d\geq r$ and $w\in A^{2d+1}$ (given by $U$) such that for all $x,y\in X$ with $x_{[-d,d]}=w=y_{[-d,d]}$ and all $n\geq 0$, we have that $$T^{n}x_{[-m,m]}=T^{n}y_{[-m,m]}.$$ Then, there is $p\in \{0,\ldots ,|w|-r\}$ such that for all $x,y\in X$ satisfying $x_{[0,|w|-1]}=w=y_{[0,|w|-1]}$, we have $$T^{n}x_{[p,p+r-1]}=T^{n}y_{[p,p+r-1]}$$ for all $n\in \NN$.

	For every $k\geq 0$ we define the set $$\Omega_{k}=\{ x\in X : \exists i\leq -k ,x_{[i,i+|w|-1 ]}=w \wedge \exists j\geq k ,x_{[j,j+|w|-1 ]}=w \}.$$ The sets $\Omega_{k}$ are clearly open. Furthermore, the transitivity of $(X,T)$ implies $\Omega_{K}$ are non-empty and dense, for every $k\geq 0$. Therefore, $\bigcap_{k\geq 0}\Omega_{k}$ is a residual set.
	 We are going to show that for every $m\geq 0$ there exists $k_{m}\geq 0$ such that 
	 $$\Omega_{k_{m}}\subseteq EQ_{2^{-m}}:=\{ x\in X : \exists \d, \forall y,z\in B_{\d}(x),\forall n\geq 0,d(T^{n}y,T^{n}z)<2^{-m} \}.$$ 
Observe that for all $x,y\in \Omega_{k}$ we have that 
$$T^{n}x_{[i+p, i+p+r-1]}=T^{n}x_{[j+p, j+p+r -1]} \ \text{and} \ T^{n}y_{[i+p, i+p+r -1]} = T^{n}y_{[j+p, j+p+r -1]}.$$ 
If $x_{[i,j+|w|]}=y_{[i,j+|w|]}$, then for all $n\geq 0$ we obtain
 $$T^{n}x_{[i+p,j+p+r-1]}=T^{n}y_{[i+p,j+p+r-1]}.$$ Therefore, for every $m\geq 0$, there exists a $k_{m}\geq 0$ sufficiently large such that  $\Omega_{k_{m}}\subseteq EQ_{2^{-m}}$. Hence, $\bigcap_{k_{m}\geq 0}\Omega_{k_{m}}\subseteq \bigcap_{m\geq 0}EQ_{2^{-m}}$. This makes $\bigcap_{m\geq 0}EQ_{2^{-m}}$ a residual set. Since $EQ=\bigcap_{m\geq 0}EQ_{2^{-m}}$ we conclude that $(X,T)$ is almost equicontinuous. 
	
\end{proof}

A TDS is \textbf{minimal} if every orbit is dense. The Auslander-Yorke dichotomy states that a minimal TDS is either equicontinuous or sensitive \cite{auslander1980}. 
Now, consider the proof of Proposition \ref{dicsub}. Note that if $(X,\sigma)$ is minimal then $\Omega_k=X$. With this observation we obtain the following result. 
\begin{proposition}
	\label{prop:mindich}
Let $(X,\sigma)$ be a minimal subshift and $(X,T)$ a CA. Then, $(X,T)$ is equicontinuous if and only if is not sensitive.
\end{proposition} 
We will now study the mean versions of equicontinuity and sensitivity (mean equicontinuity is weaker than equicontinuity and sensitivity is weaker than mean sensitivity).
\begin{definition}
Let $(X,T)$ be a TDS and $x\in X$. 
\begin{enumerate}
\item The point $x$ is a \textbf{mean equicontinuity point} 
      if for every $\e >0$, there exists $\d >0$ such that if $d(x,y)<\d$, then 
      $$\limsup_{n\rightarrow \infty} \frac{\sum_{i=0}^{n}d(T^{i}x,T^{i}y)}{n+1}\leq \e .$$
      We denote the set of mean equicontinuty points by $EQ^{M}$.
\item $(X,T)$ is \textbf{mean equicontinuous} if $X=EQ^{M}$.
\item $(X,T)$ is \textbf{almost mean equicontinuous} if $EQ^{M}$ is residual.

\item $(X,T)$ is \textbf{mean sensitive} if there exists $\e >0$ such that 
      for every non-empty open set $U\subseteq X$ there exist $x,y\in U$ such that
      $$\limsup_{n\rightarrow \infty} \frac{\sum_{i=0}^{n}d(T^{i}x,T^{i}y)}{n+1}> \e .$$
\end{enumerate}
\end{definition}
Clearly every almost equicontinuous TDS is almost mean equicontinuous. There exist many almost mean equicontinuous TDSs that are not almost equicontinuous \cite{lituye,garcia2019dynamical}; none of these examples is a CA. We will later construct an almost mean equicontinuous CA that is not almost equicontinuous.

\begin{proposition}\cite[Lemma 5]{weakforms}
	\label{pro sigma}

Let $(X,T)$ be a TDS and $\e >0$. Define

$$EQ^{M}_{\e}=\{x\in X : \exists \ \d >0, \ \forall \ y,z\in B_{\d}(x), \ \limsup_{n\rightarrow \infty}\frac{\sum_{i=0}^{n}d(T^{i}y,T^{i}z)}{n+1}<\e \} .$$
Then, $EQ^{M}_{\e}$ is open and $EQ^{M}=\bigcap_{m>0} EQ^{M}_{\frac{1}{m}}$. Furthermore, $EQ^{M}$ is dense if and only if it is a residual set.

\end{proposition}





The Akin-Auslander-Berg dichotomy can also be stated for the mean versions of equicontinuity/sensitivity. 

\begin{theorem}[Mean Akin-Auslander-Berg dichotomy \cite{lituye,weakforms}]
Transitive topological dynamical systems are mean sensitive  if only if they are not almost mean equicontinuous.
\end{theorem}

In view of the previous results in this section, it is natural to ask if there is a mean version of Theorem \ref{dicsub}. Later on, we show that this question has a negative answer. First, we will give a more concrete characterization of mean equicontinuity on CA. The following proposition uses standard tools used to connect density and averages. 

\begin{proposition}\label{me}
Let $(X,T)$ be a CA and $x\in X$. Then $x$ is a mean equicontinuity point if  and only if for every $m\geq 0$ there exists  $m'\geq 0$ such that for every $y\in B_{2^{-m'}}(x)$, the set $$S_{j}:=\{ i\in \ZZ_{\geq 0} : T^{i}x_{j}\neq T^{i}y_{j} \vee T^{i}x_{-j}\neq T^{i}y_{-j} \}$$ satisfies that $$\overline{D}(S_{j})\leq \frac{1}{2^{m+2}},$$ for all $0\leq j \leq m+1$.
\end{proposition}

\begin{proof}
\begin{itemize}
\item[$\Rightarrow$:] Let us assume on the contrary that there exists $m\geq 0$ such that for all $m'\geq 0$ there exists $y\in B_{2^{-m'}}(x)$ such that $\overline{D}(S_{l})> \frac{1}{2^{m}}$ for some $0\leq l \leq m+1 $. Therefore, there exists $(n_{j})_{j\geq 0}\subseteq \ZZ_{\geq 0}$ such that 
$$\lim_{n_{j}\rightarrow \infty }\frac{\sharp (S_{l}\cap \{ 0,\ldots , n_{j}\} )}{n_{j}+1}>2^{-m}.$$
Observe that
$$\lim_{n_{j}\rightarrow \infty }\frac{\sharp (S_{l}\cap \{ 0,\ldots , n_{j}\} )}{n_{j}+1}=\lim_{n_{j}\rightarrow \infty }\frac{\sum_{i\in S_{l}\cap \{ 0,\ldots , n_{j}\} }d(T^{i}x,T^{i}y)}{n_{j}+1}.$$
Then, we obtain that 
$$\lim_{n_{j}\rightarrow \infty }\frac{\sum_{i=0}^{n_{j}}d(T^{i}x,T^{i}y)}{n_{j}+1}>2^{-m}.$$
Hence, 
$$\limsup_{n\rightarrow \infty} \frac{\sum_{i=0}^{n}d(T^{i}x,T^{i}y)}{n+1}>2^{-m}.$$
Therefore, $x$ is not a mean equicontinuity point.

\item[$\Leftarrow$:]Let us define for every $x,y\in A^{\ZZ}$ and every pair of integers $n,k\geq 0$ the set $$C_{n,k}:=\{ 0\leq i\leq n : T^{i}x_{[-k,k]}\neq T^{i}y_{[-k,k]}\}.$$ Observe that 
\begin{enumerate}
\item for every $k\geq 0$ we have that $C_{n,k}\subseteq C_{n,k+1}$;
\item for every $k\geq 0$ we have that 
\small
$$C_{n,k+1}\setminus C_{n,k} =\{ i\in [0,n] : T^{i}x_{[-k,k]}=T^{i}y_{[-k,k]} \wedge T^{i}x_{[-(k+1),k+1]}\neq T^{i}y_{[-(k+1),k+1]} \} .$$
\end{enumerate}

Now, let us assume that for every $m\geq 0$ there exists $m'\geq 0$ such that for every $y\in B_{\frac{1}{2^{m'}}} (x)$, the set $$S_{j}=\{i\geq 0 :  T^{i}x_{j}\neq T^{i}y_{j} \vee T^{i}x_{-j}\neq T^{i}y_{-j} \}$$ satisfies $\overline{D}(S_{j})\leq \frac{1}{2^{m+2}}$, for every $0\leq j\leq m+1$. 
Then,
\begin{displaymath}
\begin{array}{rl}
 & \displaystyle\limsup_{n\rightarrow \infty} \frac{\displaystyle\sum_{i=0}^{n}d(T^{i}x,T^{i}y)}{n+1}     \\
= &\displaystyle\limsup_{n\rightarrow \infty }\frac{\sharp (C_{n,0})+\displaystyle\sum_{i=1}^{\infty}\frac{1}{2^{i}}\sharp (C_{n,i}\setminus C_{n,i-1})}{n+1} \\
= &  \displaystyle\limsup_{n\rightarrow \infty }\frac{ \displaystyle\sum_{i=0}^{m+1}\frac{1}{2^{i}}\sharp (S_{i}\cap [0,n])+\displaystyle\sum_{i=m+2}^{\infty}\frac{1}{2^{i}}\sharp (C_{n,i}\setminus C_{n,i-1})}{n+1} \\
\leq &  \displaystyle\sum_{i=0}^{m+1}\frac{1}{2^{i}}\frac{1}{2^{m+2}}+ \displaystyle\limsup_{n\rightarrow \infty }\frac{\displaystyle\sum_{i=m+2}^{\infty}\frac{1}{2^{i}}\sharp (C_{n,i}\setminus C_{n,i-1})}{n+1}\\
\leq & \displaystyle\sum_{i=0}^{m+1}\frac{1}{2^{i}}\frac{1}{2^{m+2}}+ \displaystyle\sum_{i=m+2}^{\infty}\frac{1}{2^{i}}\\
\leq & \displaystyle\frac{1}{2^{m+1}}+\frac{1}{2^{m+1}} \\
= & \displaystyle\frac{1}{2^{m}}.
\end{array}
\end{displaymath}

\end{itemize}
This implies $x$ is a mean equicontinuity point. 
\end{proof}
Mean equicontinuity of CA should not be confused with equicontinuity with respect to the Besicovitch pseudometric studied in \cite{blanchard1997cellular}.
\section{Example 1: The Pacman CA}
In this section we will construct a CA that is almost mean equicontinuous but not almost equicontinuous. First we will give the formal definition of the CA, then we will give the heuristics of the map so the reader gets intuition and finally we will approach the result using a series of technical lemmas. 
We remind the reader that $Tx_i$ represents the $i$th coordinate of the point $Tx$. 

Let $A=\{ \cero,\uno,\ghost{blue} ,\key{blue},\pacman{yellow},\cinco{blue} \}$.  We define the function $T:A^{\ZZ}\rightarrow A^{\ZZ}$ locally as follows 
{\small
\begin{displaymath}
\begin{array}{ccc}
Tx_{i}& = & \displaystyle\left\{ \begin{array}{ccl}
\cero & if & (x_{i-1}\in \{ \cero,\ghost{blue},\key{blue} \} \wedge [(x_{i}\in \{ \cero , \ghost{blue}, \key{blue} \} \wedge x_{i+1}\in \{ \cero,\uno,\pacman{yellow} \} )\\ 
 & & \vee ( x_{i}=\pacman{yellow} \wedge x_{i+1}\notin \{ \uno,\cinco{blue} \} )]) \\
 & & \vee (x_{i-1}\in \{ \uno,\cinco{blue} \} \wedge [(x_{i}\in \{ \cero,\key{blue} \} \wedge x_{i+1}\in \{ \cero,\uno,\pacman{yellow} \} ) \\ 
 & & \vee (x_{i}=\pacman{yellow} \wedge x_{i+1}\notin \{ \uno,\cinco{blue} \} )]), \\
\uno & if & x_{i}\in \{ \uno, \cinco{blue} \} \wedge x_{i+1}\notin \{ \key{blue},\cinco{blue}\} , \\

\ghost{blue} & if & (x_{i-1}\in \{ \cero,\ghost{blue},\key{blue}  \} \wedge x_{i}\in \{ \cero,\ghost{blue},\key{blue}\} \wedge x_{i+1}\in \{ \ghost{blue},\cinco{blue} \} ) \\ 
 &  & \vee (x_{i-1}\in \{ \uno,\cinco{blue} \}\wedge  x_{i}\in \{ \cero,\key{blue} \} \wedge x_{i+1}=\cinco{blue} ), \\

\key{blue} & if & (x_{i+1}=\key{blue}\wedge [(x_{i-1}\in \{ \cero,\ghost{blue},\key{blue}\} \wedge x_{i}\in \{ \cero,\ghost{blue},\key{blue} \} )\\ 
 & & \vee (x_{i-1}\in \{ \uno,\cinco{blue}\} \wedge x_{i}\in \{ \cero,\key{blue} \} )]) \\
 & & \vee (x_{i-1}=\pacman{yellow} \wedge x_{i}\notin \{ \uno,\cinco{blue} \} \wedge x_{i+1}\in \{ \uno,\cinco{blue} \} ) \\
 & & \vee (x_i=\pacman{yellow} \wedge x_{i+1}\in \{ \uno,\cinco{blue} \}),\\

\pacman{yellow} & if &  (x_{i-1}=\uno \wedge [(x_{i}\in \{ \cero,\key{blue} \} \wedge x_{i+1}=\ghost{blue}) \vee x_{i}=\ghost{blue}]) \\
 & & \vee (x_{i-1}= \pacman{yellow} \wedge x_{i},x_{i+1}\notin \{ \uno,\cinco{blue} \}) \text{, and} \\ 
 
\cinco{blue} & if & x_{i}\in \{ \uno , \cinco{blue} \} \wedge x_{i+1}\in \{ \key{blue},\cinco{blue} \} .\\

\end{array}\right.

\end{array}
\end{displaymath}
}
This CA has memory and anticipation 1.
We will call the members of the alphabet as follows: 
\begin{itemize}
\item \cero \ empty space,
\item \uno \ empty door,
\item \pacman{yellow} \ pacman,
\item \ghost{blue} \ ghost.
\item \key{blue} \ keymaster ghost, and
\item \cinco{blue} \ door with ghost.

\end{itemize}
We will now explain the heuristics of this map so the reader gets intuition on the dynamics. The reader does not need to know the rules of the game \emph{Pacman}. It is only needed to understand that pacmans eat blue ghosts. 
\begin{itemize}
\item A door always stays fixed in the same place (a ghost might cross it); that is,  $x_i\in \{\uno, \cinco{blue}\}$ if and only if $Tx_i\in \{\uno, \cinco{blue}\}$.

\item Pacmans \pacman{yellow} move to right (one position per unit of time) if there is no door; that is, if $x_i=\pacman{yellow}$ and  $x_{i+1},x_{i+2}\notin \{\uno, \cinco{blue}\}$ then $Tx_{i+1}=\pacman{yellow}$.
\item If a pacman encounters a door (on the right) it is transformed into keymaster ghost \key{blue}; that is, if $x_i=\pacman{yellow}$ and $x_{j}\in \{\uno, \cinco{blue}\}$ with $j\in \{i+1,i+2\}$ then $Tx_{j-1}=\key{blue}$.
\item Ghosts (\ghost{blue},\key{blue}) always move to the left (one position per unit of time) if there is no pacman or a door on the left; that is, if $x_{i}=  \ghost{blue} ( x_{i}=\key{blue})$, $x_{i-1}\in \{ \cero , \ghost{blue}, \key{blue} \}$ and $x_{i-2}\in \{ \cero,  \ghost{blue}, \key{blue}\} (x_{i-2}\neq  \pacman{yellow} )  $, then $Tx_{i-1}=\ghost{blue}(Tx_{i-1}=\key{blue})$. 
\item If a ghost or keymaster ghost encounters a pacman (on the left) it will dissapear (get eaten); that is, 
\begin{itemize}

\item if $x_{i}\in \{\ghost{blue}, \key{blue}\}$ and $x_{i-1}=\pacman{yellow}$, then $Tx_{i-1}\notin \{\ghost{blue}, \key{blue}\}$; and 

\item if $x_{i}\in \{\ghost{blue}, \key{blue}\}$, $x_{i-2}=\pacman{yellow}$ and $x_{i-1}\notin \{ \uno ,\cinco{blue} \}$, then $Tx_{i-1}=\pacman{yellow}$. 

\end{itemize}

\item If a ghost \ghost{blue} encounters a door it transforms into a pacman; that is, 
\begin{itemize}
\item if $x_{i}=\ghost{blue}$ and $x_{i-1}\in \{ \uno , \cinco{blue}\}$, then $Tx_{i}=\pacman{yellow}$; and
\item if $x_{i}=\ghost{blue}$, $x_{i-2}=\{ \uno , \cinco{blue}\}$ and $x_{i-1}\notin \{ \uno, \cinco{blue},\pacman{yellow}\}$, then $Tx_{i-1}=\pacman{yellow}$.
\end{itemize}

\item If a keymaster ghost encounters a door he will enter the door, lose its key, and (in the following step) proceed to the left; that is, if $x_{i}=\key{blue}$ and $x_{i-1}\in \{ \uno , \cinco{blue}\}$, then $Tx_{i-1}=\cinco{blue}$ and
\begin{itemize}
\item if $x_{i-3}=\cero$, then $T^{2}x_{i-2}=\ghost{blue}$ and
\item if $x_{i-3}=\pacman{yellow}$, then $T^{2}x_{i-2}=\key{blue}$.
\end{itemize} 

\end{itemize}
When describing a point in $A^\ZZ$ we will use a point (.) to indicate the \emph{zero}th coordinate, for example if $x= \ ^{\infty}\cero.\pacman{yellow} \ \cero^{\infty}$ then $x_0=\pacman{yellow}$ and $x_i=\cero$ for every $i\neq 0$.
We will now provide some examples on how the Pacman CA works. Notice that time flows downward on the diagrams. 

\begin{example}\label{example 2}
	Let $m\geq 2$, $w=\uno \ \cero^{m} \ \uno$. We will show a section of the orbit of $x:=  ^{\infty }\cero .w\key{blue}^{\infty }$. In this example we can observe that the space between two doors is acting like a some sort of ``filter", because many ghosts dissapear. 
	
	\begin{displaymath}
	\begin{array}{cccccccccccc}
\cero &	\uno        & \cero          &\cero          &\cero          &\cero          &\cero        &\uno        &\key{blue}&\key{blue} \\
\cero &	\uno        & \cero          &\cero          & \cero         &\cero          &\cero        &\cinco{blue}&\key{blue}&\key{blue} \\
\cero &	\uno        & \cero          &\cero          &\cero          &\cero          &\ghost{blue}&\cinco{blue}&\key{blue}&\key{blue} \\
\cero &	\uno        & \cero          &\cero          &\cero          & \ghost{blue} &\ghost{blue}&\cinco{blue}&\key{blue}&\key{blue} \\
\cero &	\uno        & \cero          &\cero          &\ghost{blue}  &\ghost{blue}  &\ghost{blue}&\cinco{blue}&\key{blue}&\key{blue} \\
\cero &	\uno        & \cero          &\ghost{blue}  &\ghost{blue}  &\ghost{blue}  &\ghost{blue}&\cinco{blue}&\key{blue}&\key{blue} \\
\cero &	\uno        & \pacman{yellow}&\ghost{blue}  &\ghost{blue}  &\ghost{blue}  &\ghost{blue}&\cinco{blue}&\key{blue}&\key{blue} \\
\cero &	\uno        & \cero          &\pacman{yellow}&\ghost{blue}  &\ghost{blue}  &\ghost{blue}&\cinco{blue}&\key{blue}&\key{blue} \\
\cero &	\uno        & \cero          &\cero          &\pacman{yellow}&\ghost{blue}  &\ghost{blue}&\cinco{blue}&\key{blue}&\key{blue} \\
\cero &	\uno        & \cero          &\cero          &\cero          &\pacman{yellow}&\ghost{blue}&\cinco{blue}&\key{blue}&\key{blue} \\
\cero &	\uno        & \cero          &\cero          &\cero          &\cero          &\key{blue} &\cinco{blue}&\key{blue}&\key{blue} \\
\cero &	\uno        & \cero          &\cero          &\cero          &\key{blue}   &\ghost{blue}&\cinco{blue}&\key{blue}&\key{blue}\\
\cero &	\uno        &\cero          &\cero          &\key{blue}   &\ghost{blue}  &\ghost{blue}&\cinco{blue}&\key{blue}&\key{blue}\\
\cero &	\uno        &\cero          &\key{blue}   &\ghost{blue}  &\ghost{blue}  &\ghost{blue}&\cinco{blue}&\key{blue} &\key{blue}\\
\cero &	\uno        &\key{blue}   &\ghost{blue}  &\ghost{blue}  &\ghost{blue}  &\ghost{blue}&\cinco{blue}&\key{blue}&\key{blue}\\
\cero &	\cinco{blue}&\pacman{yellow}&\ghost{blue}  &\ghost{blue}  &\ghost{blue}  &\ghost{blue}&\cinco{blue}&\key{blue}&\key{blue}  \\
	
	\end{array}
	\end{displaymath}
	
\end{example}

\begin{example}\label{example 1}
Let $w=\uno \ \cero \ \cero \ \ghost{blue} \ \pacman{yellow} \ \cero \ \uno \ \key{blue}$. We show a section of the orbit of $x= \ ^{\infty}\cero.w\cero^{\infty}$.
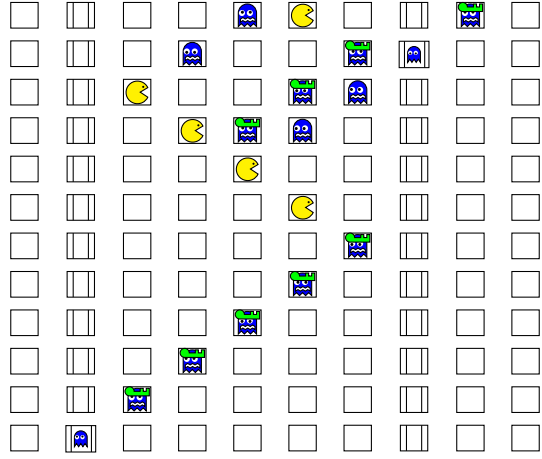
\begin{displaymath}
\begin{array}{cccccccccccc}
\cero & \uno        & \cero          &\cero          &\ghost{blue}  &\pacman{yellow}&\cero        &\uno        &\key{blue}&\cero \\
\cero & \uno        & \cero          &\ghost{blue}  & \cero         &\cero          &\key{blue} &\cinco{blue}&\cero &\cero \\ 
\cero & \uno        & \pacman{yellow}&\cero          &\cero          &\key{blue}   &\ghost{blue}&\uno      &\cero & \cero\\  
\cero & \uno        & \cero          &\pacman{yellow}&\key{blue}   & \ghost{blue} &\cero        &\uno        &\cero &\cero\\
\cero & \uno        & \cero          &\cero          &\pacman{yellow}&\cero          &  \cero      &\uno        &\cero &\cero\\
\cero & \uno        & \cero          &\cero          &\cero          &\pacman{yellow}&\cero        &\uno        &\cero &\cero \\
\cero & \uno        & \cero          &\cero          &\cero          &\cero          &\key{blue} &\uno &\cero &\cero\\
\cero & \uno        & \cero          &\cero          &\cero          &\key{blue}   &\cero        &\uno &\cero &\cero\\
\cero & \uno        & \cero          &\cero          &\key{blue}   &\cero          &\cero        &\uno &\cero &\cero\\
\cero & \uno        & \cero          &\key{blue}   &\cero          &\cero          &\cero        &\uno &\cero &\cero\\
\cero & \uno        & \key{blue}   &\cero          &\cero          &\cero          &\cero        &\uno &\cero &\cero\\
\cero & \cinco{blue}& \cero          &\cero          &\cero          &\cero          &\cero        &\uno &\cero &\cero
\end{array}
\end{displaymath}
\end{example}

\
We now prove a series of technical lemmas. If there is one \key{blue} to the right of a pattern with empty spaces and empty doors, then  \key{blue}  will ``cross" all the doors eventually. We state this fact formally in the next lemma.

\begin{lemma}\label{palabrafinita01}
Let $m\geq 1$, $w\in \{ \cero ,\uno \}^{m}$ such that $w_{0}=\uno=w_{m-1}$ and $w_{i}=\cero$ for all $0<i<m-1$. Set $x=^{\infty }\cero .w\key{blue} \ \cero^{\infty }$. Then, there exists $N>0$ such that $T^{N}x_{[0,m-1]}=w$.
\end{lemma}

\begin{proof}
 Assume that $m\geq 4$. From the definition of $T$ we have the followings implications: 
\begin{itemize}
\item $Tx_{m-1}=\cinco{blue}$ $\wedge$  $Tx_{i}\in \{ \cero , \uno \}$ $\forall$ $i\neq m-1$, 
\item $T^{m-j}x_{j}=\ghost{blue}$ for $1<j<m-1$ $\wedge$ $T^{m-j}x_{i}\in \{ \cero , \uno \}$ $\forall$ $i\neq j$,
\item $T^{m-2+j}x_{j}=\pacman{yellow}$ for $1 \leq j <m-2$ $\wedge$ $T^{m-2+j}x_{i}\in \{ \cero , \uno \}$ $\forall$ $i\neq j$,
\item $T^{3m-6-j}x_{j}=\key{blue}$ for $1\leq j\leq m-2$ $\wedge$ $T^{3m-6-j}x_{i}\in \{ \cero , \uno \}$ $\forall$ $i\neq j$,
\item $T^{3m-6}x_{0}=\cinco{blue}$  $\wedge$ $T^{3m-6}x_{i}\in \{ \cero , \uno \}$ $\forall$ $i\neq 0$.
\end{itemize}
Then, for $N=3m-5$, we have that $T^{N}x_{[0,m-1]}=w$.
The case when $1\leq m\leq 3$ is easy to check.
\end{proof}

\begin{remark} \label{Remark1}
	Note that if $x_i=\uno$ and $x_i+1=\cero$ then $Tx_i=\uno$.
\end{remark}
 Using Remark \ref{Remark1} and Lemma \ref{palabrafinita01} we obtain the following result. 
\begin{lemma}\label{lemma3}
Let $m>0$, $w\in A^{m}$  and $x=^{\infty }\cero .w\uno \ \cero^{\infty}$. There exists $N>0$ such that for all $n\geq N$,
\begin{center}
$T^{n}x_{i}\in \{ \cero,\uno \} \ \forall \  i\geq 0.$
\end{center}
\end{lemma}

In the proof of Lemma \ref{palabrafinita01} we describe the ``trajectory" of \key{blue} from the start until it crosses the doors. In the follwing lemma we describe a similar trajectory, but this time we are going to do it backwards in time.

\begin{lemma}\label{trayectoria}
Let $m\geq 2$, $v\in A^{\NN}$ and $x:= \infinito \cero . \uno \ \cero^{m} \ \uno v$. If $N\geq 3m$ and $T^{N}x_{0}=\cinco{blue}$, then
\begin{itemize}
\item $T^{N-3m+1}x_{m+1}=\cinco{blue}$, 
\item $T^{N-2(m-1)-j}x_{j}=\ghost{blue}$ for $2\leq j\leq m$,
\item $T^{N-m-j}x_{m-j}=\pacman{yellow}$ for $1 \leq j \leq m-1$, and
\item $T^{N-j}x_{j}=\key{blue}$ for $1\leq j\leq m$.
\end{itemize}
\end{lemma}
\begin{proof}
	Assume the hypothesis of the lemma. By checking the rules of $T$ one can see that if $T^{N}x_{0}=\cinco{blue}$ and $x_1\neq \uno$ then necessarily $T^{N-1}x_{1}=\key{blue}$. We can go back step by step to obtain the result.  
\end{proof}

Using Lemma \ref{trayectoria} we will se that if $T^{N}x_{0}=\cinco{blue}=T^{N'}x_{0}$, then $N$ and $N'$ cannot be near.

\begin{lemma}\label{Lcontradiccion}

Let $m\geq 0$, $v\in A^{\NN}$, and $x:= \infinito\cero . \uno \ \cero^{m}\uno v$. If $N' >N\geq 3m$ are such that $T^{N}x_{0}=\cinco{blue}=T^{N'}x_{0}$, then:
\begin{itemize}
\item if $0\leq m\leq 1$, then $ N'-N >2m$; and
\item if $m\geq 2$, then $ N'-N \geq 2m-1$.
\end{itemize}
\end{lemma}
\begin{proof}
The case $0\leq m\leq 1$ is trivial.\\


Let $m\geq 2$, and $N' >N\geq 3m$ such that $T^{N}x_{0}=\cinco{blue}=T^{N'}x_{0}$. Assume that $N'-N<2m-1$. From Lemma \ref{trayectoria} we have that 
\begin{itemize}
\item $T^{N-j}x_{j}=\key{blue}$ for $1\leq j\leq m$ and
\item  $T^{N'-m-j'}x_{m-j'}=\pacman{yellow}$ for $1 \leq j' \leq m-1$.
\end{itemize}
First, suppose that $N'-N$ is even. Let $j=m-\frac{N'-N}{2}$ and $j'=\frac{N'-N}{2}$. By the asumption on $N,N'$ and $m$ it follows  that $1 \leq j \leq m$, $1 \leq j' \leq m-1$, and $$T^{N-j}x_{j}=T^{N'-m-j'}x_{m-j};$$ a clear contradiction. 

Now suppose $N'-N$ is odd. Let $j=m-\lceil \frac{N'-N}{2}\rceil$ and $j'=\lceil \frac{N'-N}{2}\rceil$. By the asumption on $N,N'$ and $m$ it follows that $1 \leq j \leq m$, $1 \leq j' \leq m-1$,

$$T^{N-j}x_{j}=\key{blue} \ \text{, and} \ T^{N'-m-j'}x_{m-j'}=\pacman{yellow}.$$
Therefore, 
$$ T^{N-j}x_{[j,j+1]}= \key{blue}  \pacman{yellow}.$$
This is also a contradiction because $\key{blue}  \pacman{yellow}$ is not on the image of $T$. 


\end{proof}

In Example \ref{example 2}, we see that considering an infinite right-tail of keymater ghosts, some get eaten and some  cross the doors. It is natural to ask what will be the frequency of \key{blue} that cross a doors. The next lemma answers this question.

\begin{lemma} \label{cota1}
	Let $w=\uno \ \cero^{m} \uno$, with $m\geq 0$ and $x=^{\infty }\cero .w\key{blue}^{\infty }$. 
	
	If $0\leq m\leq 1$, then
	\begin{itemize}
		\item  $T^{3m-2}x_{0}= \cinco{blue}$,
		\item $T^{3m-2+(2m+1)k}x_{0}= \cinco{blue}$ for $k\geq 0$, and
		\item $T^{i}x_{0}=\uno$, for all $3m-2+(2m+1)k <i< 3m-2+(2m+1)(k+1)$ and  $k\geq 0$. 
	\end{itemize}	
	If $m\geq 2$, then 
	\begin{itemize}
	    \item $T^{3m}x_{0}= \cinco{blue}$, 
	    \item $T^{3m+(2m-1)k}x_{0}= \cinco{blue}$ for  $k\geq 0$,  and 
	    \item $T^{i}x_{0}=\uno$, for all $3m+(2m-1)k<i<3m+(2m-1)(k+1)$ and 
	    $k\geq 0$.  
\end{itemize}
\end{lemma}

\begin{proof}
The proof for $0\leq m\leq 1$ is similar to the proof when $m\geq 2$. So we are only going to prove the result when $m\geq 2$. Using a similar argument of the proof of Lemma~\ref{palabrafinita01} we obtain that $T^{3m}x_{0}= \cinco{blue}$  and  $T^{i}x_{0}= \uno$ for all $0<i<3m$. Also, we have that $T^{2m-1}x_{m+2}=\key{blue}$. Hence, $T^{5m-1}x_{0}= \cinco{blue}$.\\

We will proceed by induction on $k$. Let us assume that $$T^{3m+(2m-1)l}x_{0}=\cinco{blue}.$$ Next, let $k=l+1$. By the induction hypothesis we have that 
$$T^{2m-1+(2m-1)l}x_{(m+2)}=\key{blue}.$$ 
Hence, $T^{5m - 1 +(2m-1)l}x_{0}=\cinco{blue}.$ Doing simple calculations we obtain $$T^{3m + (2m-1)(l+1)}x_{0}=\cinco{blue}.$$ 

The proof of $T^{i}x_{0}=\uno$, for all $3m+(2m-1)k<i<3m+(2m-1)(k+1)$ and  $k\geq 0$, follows immediately from Lemma \ref{Lcontradiccion}. 
\end{proof}

We will now prove that the set of equicontinuity points is empty.

\begin{proposition} \label{finiteword}

Let $m\geq 1$ and $w\in A^{m}$. Then there exist $x,y\in A^{\ZZ}$ such that $$x_{[0,m-1]}=w=y_{[0,m-1]}$$ and the set $$S=\{ i\in \ZZ_{\geq 0} : T^{i}x_{0}\neq T^{i}y_{0} \}$$ is infinite.
\end{proposition}

\begin{proof}
Let $m\geq 1$, $w\in A^{m}$ and $x=^{\infty}\cero .w\cero^{\infty}$. Lemma \ref{lemma3} says there exists $N>0$ such that $T^{n}x_{0}\in \{ \cero,\uno \}$ for every $n\geq N$. Let $y=^{\infty}\cero .w\uno \ \key{blue}^{\infty}$, by Lemma \ref{cota1} the the set $S$ is infinite.
\end{proof}

Lemma \ref{cota1} tells us the exact frecuency of \key{blue} crossing doors when we have a tail $\key{blue}^{\infty}$ to the right. If we do not have precise information on what is in the right we may not have the exact frecuency as in the Lemma \ref{cota1}. However, using Lemma \ref{Lcontradiccion}, we will be able to obtain an upper bound.






Now we will explore a similar situation but with finitely many doors. 

\begin{lemma}\label{lemma2}
Let $\{d_i \}_{i=0}^n$ a finite set of non-negative integers, $v\in A^{\NN}$, $$w=\uno \cero^{d_{0}}\uno \cero^{d_{1}} \cdots \uno \cero^{d_{n}} \uno,$$ $x=\infinito\cero .wv,$ and $0\leq j< n+\sum_{i=0}^{n-1}d_{i}$. Assume that $T^{N}x_{j}= \cinco{blue} = T^{N'}x_{j}$ for some $N,N'\geq 0$.
We have that
\begin{itemize}
\item if $0\leq d_{n}\leq 1$, then $|N-N'|>2d_{n}$, and
\item if $d_{n}\geq 2$, then $|N-N'|>2(d_{n}-1)$. 
\end{itemize}
  
\end{lemma}

\begin{proof}
The case where $n=0$ is a direct application of Lemma \ref{Lcontradiccion}. We will prove the other case by induction. Assume that for $n=p$ the result holds. Now, let $n=p+1$ . By the induction hypothesis we have that if $x_{j}=\uno$, for all $1\leq j \leq p+1+ \sum_{i=0}^{p}d_{i}$, and $T^{N}x_{j} = \cinco{blue} = T^{N'}x_{j}$ for all $N,N'\geq 0$ then

\begin{displaymath}
\begin{array}{rccl}
 \text{if}  & d_{p+1}\geq 2       & \text{then} & \vert N-N'\vert \geq 2(d_{p+1}-1)  \ \text{or} \\
\text{if} & 0\leq d_{p+1}\leq 1 & \text{then} & \vert N-N'\vert \geq 2d_{p+1}. 
\end{array}
\end{displaymath}

Hence, the only thing left to proof is that for $x_{0}=\uno$ and all $N,N'\geq 0$ such that $T^{N}x_{0} = \cinco{blue} = T^{N'}x_{0}$ we have that 
\begin{displaymath}
\begin{array}{rccl}
\text{if} & d_{p+1}\geq 2       & \text{then} &\vert N-N'\vert \geq 2(d_{p+1}-1)\ \text{or} \\
\text{if} & 0\leq d_{p+1}\leq 1 & \text{then} &\vert N-N'\vert \geq 2d_{p+1}.
\end{array}
\end{displaymath}
For $0\leq d_{p+1}\leq 1$ the result is trivial. So, let us assume that $d_{p+1}\geq 2$. Also, let us assume that there exist $N,N'\geq 0$ such that $T^{N}x_{0} = \cinco{blue} = T^{N'}x_{0}$.
This means that there exist $N_{0},N_{0}'\geq 0$ such that $T^{N_{0}}x_{d_{0}+1} = \cinco{blue} =  T^{N_{0}'}x_{d_{0}+1}$ and
$N_{0}' + r=N'$ and $N_{0} + r=N$. Therefore,
\begin{displaymath}
\begin{array}{rcl}
2(d_{p+1}-1) & \leq & |N'-N|. \\
\end{array}
\end{displaymath}
\end{proof}

For Lemma \ref{lemma6} it will be useful to consider the CA as a (vanishing) particle system, where ghosts and pacmans are particles. 

We define the \textbf{particle function} $\gamma: A^{\ZZ}\rightarrow \{ \cero , \punto{black} \ \}^{\ZZ}$ as

\begin{displaymath}
\gamma (x)_{i}= \left\lbrace 
\begin{array}{lcl}
\cero & if & x_{i}\in \{ \cero, \uno \} , \\
\punto{black} \ & if & x_{i} \in \{ \ghost{blue},\key{blue},\pacman{yellow}, \cinco{blue} \} , 
\end{array} \right.  
\end{displaymath}
where $x\in A^{\ZZ}$ and $i\in \ZZ$. Observe that with this function the Examples \ref{example 2} and \ref{example 1} turn out as follows:

\begin{displaymath}
\begin{array}{ccccccccccccccccccccccc}
\uno        & \cero          &\cero          &\cero          &\cero          &\cero        &\uno        &\key{blue}&\key{blue} & & 
\cero       & \cero          &\cero          &\cero          &\cero          &\cero        &\cero       &\punto{black}     &\punto{black}      \\
\uno        & \cero          &\cero          & \cero         &\cero          &\cero        &\cinco{blue}&\key{blue}&\key{blue} & &
\cero       & \cero          &\cero          & \cero         &\cero          &\cero        &\punto{black}     &\punto{black}     &\punto{black}     \\
\uno        & \cero          &\cero          &\cero          &\cero          &\ghost{blue}&\cinco{blue}&\key{blue}&\key{blue} & &
\cero       & \cero          &\cero          &\cero          &\cero          &\punto{black}     &\punto{black}     &\punto{black}     &\punto{black}       \\
\uno        & \cero          &\cero          &\cero          & \ghost{blue} &\ghost{blue}&\cinco{blue}&\key{blue}&\key{blue} & &
\cero       & \cero          &\cero          &\cero          & \punto{black}      &\punto{black}      &\punto{black}     &\punto{black}     &\punto{black}      \\
\uno        & \cero          &\cero          &\ghost{blue}  &\ghost{blue}  &\ghost{blue}&\cinco{blue}&\key{blue}&\key{blue} & &
\cero       & \cero          &\cero          &\punto{black}        &\punto{black}        &\punto{black}      &\punto{black}    &\punto{black}     &\punto{black}  \\
\uno        & \cero          &\ghost{blue}  &\ghost{blue}  &\ghost{blue}  &\ghost{blue}&\cinco{blue}&\key{blue}&\key{blue} & &
\cero       & \cero          &\punto{black}        &\punto{black}        &\punto{black}        &\punto{black}      &\punto{black}     &\punto{black}     &\punto{black}      \\
\uno        & \pacman{yellow}&\ghost{blue}  &\ghost{blue}  &\ghost{blue}  &\ghost{blue}&\cinco{blue}&\key{blue}&\key{blue} & & 
\cero       & \punto{black}        &\punto{black}        &\punto{black}        &\punto{black}        &\punto{black}      &\punto{black}     &\punto{black}     &\punto{black}      \\
\uno        & \cero          &\pacman{yellow}&\ghost{blue}  &\ghost{blue}  &\ghost{blue}&\cinco{blue}&\key{blue}&\key{blue} & \Rightarrow & 
\cero       & \cero          &\punto{black}        &\punto{black}      &\punto{black}       &\punto{black}      &\punto{black}     &\punto{black}     &\punto{black}      \\
\uno        & \cero          &\cero          &\pacman{yellow}&\ghost{blue}  &\ghost{blue}&\cinco{blue}&\key{blue}&\key{blue} & &
\cero       & \cero          &\cero          &\punto{black}        &\punto{black}        &\punto{black}      &\punto{black}     &\punto{black}     &\punto{black}      \\
\uno        & \cero          &\cero          &\cero          &\pacman{yellow}&\ghost{blue}&\cinco{blue}&\key{blue}&\key{blue} & & 
\cero       & \cero          &\cero          &\cero          &\punto{black}        &\punto{black}      &\punto{black}  &\punto{black}     &\punto{black}      \\
\uno        & \cero          &\cero          &\cero          &\cero          &\key{blue} &\cinco{blue}&\key{blue}&\key{blue} & & 
\cero       & \cero          &\cero          &\cero          &\cero          &\punto{black}      &\punto{black}     &\punto{black}    &\punto{black}       \\
\uno        & \cero          &\cero          &\cero          &\key{blue}   &\ghost{blue}&\cinco{blue}&\key{blue}&\key{blue} & &
\cero       & \cero          &\cero          &\cero          &\punto{black}        &\punto{black}      &\punto{black}     &\punto{black}     &\punto{black} \\
\uno        &\cero          &\cero          &\key{blue}   &\ghost{blue}  &\ghost{blue}&\cinco{blue}&\key{blue}&\key{blue} & &
\cero       &\cero          &\cero          &\punto{black}        &\punto{black}        &\punto{black}      &\punto{black}     &\punto{black}     &\punto{black}     \\
\uno        &\cero          &\key{blue}   &\ghost{blue}  &\ghost{blue}  &\ghost{blue}&\cinco{blue}&\key{blue} &\key{blue} & & 
\cero       &\cero          &\punto{black}        &\punto{black}        &\punto{black}        &\punto{black}      &\punto{black}     &\punto{black}      &\punto{black}     \\
\uno        &\key{blue}   &\ghost{blue}  &\ghost{blue}  &\ghost{blue}  &\ghost{blue}&\cinco{blue}&\key{blue}&\key{blue}& & 
\cero       &\punto{black}        &\punto{black}        &\punto{black}        &\punto{black}        &\punto{black}      &\punto{black}     &\punto{black}     &\punto{black}     \\
\cinco{blue}&\pacman{yellow}&\ghost{blue}  &\ghost{blue}  &\ghost{blue}  &\ghost{blue}&\cinco{blue}&\key{blue}&\key{blue}  & &
\punto{black}     &\punto{black}        &\punto{black}        &\punto{black}        &\punto{black}       &\punto{black}      &\punto{black}     &\punto{black}     &\punto{black}      \\

\end{array}
\end{displaymath}

\begin{displaymath}
\begin{array}{ccccccccccccccccccccccc}
\uno        & \cero          &\cero          &\ghost{blue}  &\pacman{yellow}&\cero        &\uno        &\key{blue}&\cero & & \cero        & \cero          &\cero          &\punto{black}  &\punto{black}  &\cero        &\cero        &\punto{black}  &\cero\\

\uno        & \cero          &\ghost{blue}  & \cero         &\cero          &\key{blue} &\cinco{blue}&\cero &\cero & &\cero        & \cero          &\punto{black}  & \cero         &\cero          &\punto{black}  &\punto{black}  &\cero &\cero \\ 

\uno        & \pacman{yellow}&\cero          &\cero          &\key{blue}   &\ghost{blue}&\uno      &\cero & \cero & & \cero        & \punto{black}  &\cero          &\cero          &\punto{black}   &\punto{black}  &\cero      &\cero & \cero \\  

\uno        & \cero          &\pacman{yellow}&\key{blue}   & \ghost{blue} &\cero        &\uno        &\cero &\cero &         & \cero        & \cero          &\punto{black}     &\punto{black}        & \punto{black}       &\cero        &\cero       &\cero &\cero \\

\uno        & \cero          &\cero          &\pacman{yellow}&\cero          &  \cero      &\uno        &\cero &\cero & &\cero       & \cero          &\cero          &\punto{black}      &\cero          &  \cero      &\cero       &\cero &\cero\\

\uno        & \cero          &\cero          &\cero          &\pacman{yellow}&\cero        &\uno        &\cero &\cero & \Rightarrow &      \cero       & \cero          &\cero          &\cero          &\punto{black}        &\cero        &\cero       &\cero &\cero \\

\uno        & \cero          &\cero          &\cero          &\cero          &\key{blue} &\uno &\cero &\cero & &            \cero      & \cero          &\cero          &\cero          &\cero          &\punto{black}      &\cero &\cero &\cero\\

\uno        & \cero          &\cero          &\cero          &\key{blue}   &\cero        &\uno &\cero &\cero & &            \cero       & \cero          &\cero          &\cero          &\punto{black}     &\cero        &\cero &\cero &\cero\\

\uno        & \cero          &\cero          &\key{blue}   &\cero          &\cero        &\uno &\cero &\cero & &           \cero       & \cero          &\cero          &\punto{black}      &\cero          &\cero        &\cero &\cero &\cero\\

\uno        & \cero          &\key{blue}   &\cero          &\cero          &\cero        &\uno &\cero &\cero & &              \cero       & \cero          &\punto{black}      &\cero          &\cero          &\cero        &\cero &\cero &\cero\\

\uno        & \key{blue}   &\cero          &\cero          &\cero          &\cero        &\uno &\cero &\cero & &            \cero       & \punto{black}      &\cero          &\cero          &\cero          &\cero        &\cero &\cero &\cero\\
\cinco{blue}& \cero          &\cero          &\cero          &\cero          &\cero        &\uno &\cero &\cero & &              \punto{black}    & \cero          &\cero          &\cero          &\cero          &\cero        &\cero &\cero &\cero
\end{array}
\end{displaymath}

Given $x,y\in X$, we define the sets $$S_{+j}:=\{ i\geq 0: T^{i}x_{j}\neq T^{i}y_{j}  \}$$ and $$S_{-j}:=\{ i\geq 0: T^{i}x_{-j}\neq T^{i}y_{-j}  \}.$$ Observe that $S_{j}=S_{+j}\cup S_{-j}$ (see Proposition \ref{me} for the defintion of $S_{j}$). 

\begin{lemma}\label{lemma6}
Let $d>0$, $w=\uno \ \cero^{d}\uno$, $x:=^{\infty}\cero .w\cero^{\infty}$, $v\in A^{\NN}$, and $y=^{\infty}\cero .wv$. If $1\leq i\leq d$, then
$$3\overline{D}(S_{+(d+1)})\geq \overline{D}(S_{+i}).$$
\end{lemma}
\begin{proof}
	Using the CA and the particle function we can define a \emph{trajectory function} of an specific particle (a ghost/pacman) $p$. We will not construct this function explicitely, but we will give its properties.
	Given a point $x\in X$ and a particle of that point; that is, $p\in \ZZ$ with $\gamma(x)_p=\punto{black}$, we can define trajectory of that particle (all the way to the infinity or until it dissapears). This trajectory is a function $\tau_p:N \rightarrow \ZZ$ where $N\subset \NN$ is the lifespan of the particle ($N=\NN$ if it never dissapears), $\tau_p (0)=p$ and $\tau_p (n)$ the position at time $n$. We have that $|\tau_p(n)-\tau_p(n+1)|\leq 1$ for $n+1\in N$. Using the properties of $T$ it is not hard to see that for every $z\in \ZZ$ we have that $|\tau_p^{-1}(z)|\leq 3$, that is, a particle can only be at most three times on a particular position. By Lemma \ref{lemma3}, there exists $N>0$ such that if for some $l>0$, $T^{N+l}y_{i}\notin \{ \cero, \uno \}$, then there exists a unique $k\geq |w|$, such that $\gamma (y)_{k}=\punto{black}$ and $\tau_{k}(N+l)=i$. Hence, for all $n\in S_{+i}$, there exist a unique $k_{n}\geq |w|$ and $m<n$ such that $\tau_{k_{n}}(n)=i$ and $\tau_{k_{n}}(m)=d+1$. 
Since $y_{d+1}=\uno$, then $|\tau_{k_{n}}^{-1}(d+1)|=1$. Define $P:=\{ z\in \ZZ : \gamma(y)_{z}=\punto{black} \}$.
Therefore,
\begin{equation} 
\begin{array}{rcl}
\limsup_{n\rightarrow \infty}\frac{3\sharp[(\bigcup_{z\in P}\tau_{z}^{-1}(d+1))\cap 
[0,n]]}{n+1} & \geq & \limsup_{n\rightarrow \infty}\frac{\sharp[(\bigcup_{z\in P}\tau_{z}^{-1}(i))\cap [0,n]]}{n+1}.
\end{array}\label{eq1}
\end{equation}
Since $\gamma (x)_{z}=\cero$ for all $z\in \ZZ$, we have that  
$$\displaystyle\bigcup_{z\in P}\tau_{z}^{-1}(i)=S_{+i} \ \text{and}$$
$$\displaystyle\bigcup_{z\in P}\tau_{z}^{-1}(d+1)=S_{+(d+1)}.$$
Therefore, from (\ref{eq1}), we conclude that
$$
3\overline{D}(S_{+(d+1)}) \geq  \overline{D}(S_{+i}).
$$
\end{proof}

\begin{proposition}\label{punto-m-e}

Let $x= \cdots \uno \ \cero^{2^{2}} \uno \ \cero^{2^{1}} \uno \ \cero^{2^{0}} . \uno \ \cero^{2^{0}} \uno \ \cero^{2^{1}}\uno \ \cero^{2^{2}} \cdots $. Then, $x$ is a mean equicontinuity point.

\end{proposition}

\begin{proof}
We are going to divide the proof in two parts:

Part 1: Let $m\geq 0$, $m'=m+3+\sum_{l=0}^{m+3} 2^{l}$, and $y\in X$ with
$$y_{[-k,k]}= \uno \ \cero^{2^{m+3}} \cdots \uno \ \cero^{2^{2}} \uno \ \cero^{2^{1}}\uno \ \cero^{2^{0}}.
 \uno \ \cero^{2^{0}} \uno \ \cero^{2^{1}}\uno \ \cero^{2^{2}} \cdots  \uno \ \cero^{2^{m+3}} \uno,$$
 for a certain $k$.
By Lemma \ref{lemma2}, we have that if $x_{j}=$\uno, where $0\leq j\leq m+3+\sum_{l=0}^{m+2}2^{l}$, then for all $N,N'\geq 0$ such that $T^{N}y_{j} = \cinco{blue} = T^{N'}y_{j}$, satisfies $| N-N' |\geq 2(2^{m+3}-1)$. Now,
\begin{displaymath}
\begin{array}{rcl}
\displaystyle\limsup_{n\rightarrow \infty} \frac{\sharp (S_{j}\cap [0,n])}{n+1} &  
 \leq &\displaystyle\limsup_{n\rightarrow n} \frac{\sharp (S_{+j}\cap [0,n])}{n+1} +\displaystyle\limsup_{n\rightarrow n} \frac{\sharp (S_{-j}\cap [0,n])}{n+1}.

\end{array}
\end{displaymath}
Define $N_{0}=\min S_{+j}$. Observe that
$\frac{\sharp (S_{+j}\cap [0,N_{0}])}{N_{0}+1}=\frac{1}{N_{0}+1}$. Let $N_{1}=\min S_{+j}\setminus \{ N_{0}\}$. We have that $\frac{\sharp (S_{+j}\cap [0,N_{1}])}{N_{1}+1}=\frac{2}{N_{1}+1}<\frac{2}{N_{0}+2(2^{m+3}-1)+1}$. 
Following this construction, for every $r\geq 1$, we define $N_{r}=\min (S_{+j}\setminus \{ N_{l} \}_{l=0}^{r-1})$. Observe that
\begin{displaymath}
\begin{array}{rcl}
\frac{\sharp (S_{+j}\cap [0,N_{r}])}{N_{r}+1} & = & \frac{r+1}{N_{r}+1}\\
& < & \frac{r+1}{r(2^{m+4}-2+\frac{1}{r})}.

\end{array}
\end{displaymath}
Since $$\lim_{r\rightarrow \infty}\frac{r+1}{r} \frac{1}{2^{m+4}-2+\frac{1}{r}}=\frac{1}{2^{m+4}-2},$$ then 
$$\lim_{r\rightarrow \infty} \frac{\sharp (S_{+j}\cap [0,N_{r}])}{N_{r}+1}\leq \frac{1}{2(2^{m+3}-1)}<\frac{1}{2^{m+3}}.$$
Similarly, we obtain
$$\lim_{r\rightarrow \infty} \frac{\sharp (S_{-j}\cap [0,N_{r}])}{N_{r}+1}<\frac{1}{2^{m+3}}.$$
Thus, 
$$\lim_{r\rightarrow \infty} \frac{\sharp (S_{j}\cap [0,N_{r}])}{N_{r}+1}\leq \frac{1}{2^{m+2}}.$$

Part 2: By Lemma \ref{lemma6} and Part 1, we have that for all $j\in \ZZ$ with $x_{j}= \cero$, and $$-(m+2+\sum_{l=0}^{m+2}2^{l}) \leq j\leq m+2+\sum_{l=0}^{m+2}2^{l},$$ then
$$3\overline{D}(S_{d})\geq \overline{D}(S_{j}),$$
where $d=m+3+\sum_{l=0}^{m+2}2^{l}$.
Since $\overline{D}(S_{d})\leq \frac{1}{3}\frac{1}{2^{m+2}},$ then $\overline{D}(S_{j})\leq \frac{1}{2^{m+2}}$. Therefore, Proposition \ref{me} gives us that $x$ is a mean equicontinuity point.

\end{proof}

The proof of Lemma  \ref{punto-m-e-lemma} is very similar to the proof of Lemma \ref{punto-m-e}. 

\begin{lemma}\label{punto-m-e-lemma}

Let $m>0$, $w\in A^{m}$ and 
$$x:= \cdots \uno \ \cero^{2^{2}} \uno \ \cero^{2^{1}} \uno \ \cero^{2^{0}} .w \uno \ \cero^{2^{0}} \uno \ \cero^{2^{1}}\uno \ \cero^{2^{2}} \cdots .$$ 
We have that $x$ is a mean equicontinuity point.

\end{lemma}

\begin{theorem}
\label{teorema1}
 $(A^{\ZZ},T)$ has no equicontinuity points (hence is not almost equicontinuous). However, it is almost mean equicontinuous.
\end{theorem}

\begin{proof}
The first statement follows immediately from Proposition \ref{finiteword}.\\

\noindent Now, let $x\in A^{\ZZ}$, $m\geq 0$, and $w=x_{[0,m]}$. We set $$y:=\cdots \uno \ \cero^{2^{2}} \uno \ \cero^{2^{1}} \uno \ \cero^{2^{0}} .w \uno \ \cero^{2^{0}} \uno \ \cero^{2^{1}}\uno \ \cero^{2^{2}} \cdots .$$ From Lemma \ref{punto-m-e-lemma}, we conclude that $y$ is a mean equicontinuity point. Therefore, $(A^{\ZZ},T)$ is almost mean equicontinuous.  
\end{proof}

\section{Example 2: The Pacman level 2 CA}
	
Let $A=\{ \cero, \uno, \ghost{blue}, \key{blue}, \pacman{yellow}, \cinco{blue} \}$, $A_{2}=\{ \cero, \cherry , \banana \}$ and $T:A^{\ZZ}\rightarrow A^{\ZZ}$ the Pacman CA of Section 3. We define $T_{2}:A_{2}^{\ZZ}\rightarrow A_{2}^{\ZZ}$ as
\begin{displaymath}
\begin{array}{rcl}
T_{2}x_{i} & = & \left\lbrace \begin{array}{lcl}
\cero & if & x_{i}=\cero ; \\
\cherry  & if & x_{i}=\banana ;\\
\banana & if & x_{i}=\cherry.
\end{array}  
\right.  
\end{array}
\end{displaymath}

\noindent
Now we will define some sort of skew product. 
 We define $A_{P}:= A\times A_{2},$ 
and the map $T_{P}:A_{P}^\ZZ \rightarrow A_{P}^\ZZ$ as

\begin{displaymath}
\begin{array}{rcl}
T_{P}x_{i} & = &\left\lbrace \begin{array}{lcl}
(Tx_{i},T_{2}x_{i}) & if & x_{i}\notin \{ (\key{blue},\cherry),(\cinco{blue}, \cherry), (\ghost{blue}, \cherry) \} ;\\
(Tx_{i},\cherry) & if & x_{i}\in \{ (\key{blue},\cherry),(\cinco{blue}, \cherry),(\ghost{blue}, \cherry) \} . 
\end{array}

\right.  
\end{array}
\end{displaymath}

\begin{lemma}\label{palabra sin fantasmas}
Let $m>0$, $w\in A_{P}^{m}$, and $$x=\infinito(\cero,\cero).w(\uno, \cero)(\cero,\cero)^{\infty}.$$ Then, there exists $N>0$ such that for all $n\geq N$ and all $0\leq i\leq |w|$,
$$T_{P}^{n}x_{i}\in \{ (p,q): p\in \{ \cero,\uno \}\wedge q\in A_{2} \} .$$

\end{lemma}

\begin{proof}
This proof follows immediately from Lemma \ref{lemma3}.
\end{proof}

We want to show that $(A_{P}^\ZZ,T_{P})$ is not almost mean equicontinuous. Using Proposition \ref{pro sigma}, we need to find a non-empty open set that does not contain any mean equicontinuity points.

\begin{lemma}\label{palabra S finito}
Let $m>0$ and $w\in A_{P}^{m}$ such that $w_{0}=(\uno,\cherry)$. Then, there exist $x,y\in A_{P}^{\ZZ}$ such that $$x_{[0,|w|-1]}=y_{[0,|w|-1]}=w$$ and the set $$\ZZ_{n\geq 0}\setminus \{n\in \ZZ_{n\geq 0} : T_{P}^{n}x_{0}\neq T_{P}^{n}y_{0} \}$$ is finite. 
\end{lemma}

\begin{proof}
Let $w\in A_{P}^{m}$ as in the hypothesis of the lemma.  Let us define $$x:=\infinito(\cero,\cero).w(\uno , \cero )(\key{blue} , \cero)(\cero, \cero)^{\infty}$$ and $$y:= \infinito(\cero,\cero).w(\uno , \cero )(\cero, \cero)^{\infty}.$$ Using Lemma \ref{palabra sin fantasmas} we can assume, without loss of generality, that $w_{i}\in \{ (p,q) : p\in \{ \cero, \uno \} \wedge q\in A_{2} \}$. Now, there exists $N>0$ such that $T_{P}^{N}x_{0}=(\cinco{blue}, q)$, where $q\in \{ \cherry,\banana \}$. Meanwhile, for all $i\geq 0$, we have that $T_P^{i}y_{0}=(\uno, q)$ with $q\in \{ \cherry,\banana \}$. We have two cases to prove.
\begin{itemize}
\item[ \ ] Case 1: $T_{P}^{N}x_{0}=(\cinco{blue}, \cherry)$. 

This implies that $T_{P}^{N+1}x_{0}=(\uno, \cherry )$. Meanwhile,  $T_{P}^{N+1}y_{0}=(\uno , \banana )$. Therefore, we can easily see that $T_{P}^{N+i}x_{0}\neq T_{P}^{N+i}y_{0}$, for all $i>0$.

\item[ \ ] Case 2: $T_{P}^{N}x=(\cinco{blue}, \banana)$.

 Again we have that $T_{P}^{N+1}x_{0}=(\uno, \cherry )$. So, $T_{P}^{N+i}x_{0}= T_{P}^{N+i}y_{0}$ for all $i\geq 0$. In this case we redefine 
$$x:=\infinito(\cero,\cero).w(\uno , \cero )(\cero,\cero)(\key{blue} , \cero)(\cero, \cero)^{\infty}$$
 and finish the proof with a similar argument as the one given in Case 1.  
\end{itemize} 
\end{proof}

\begin{lemma}\label{punto no m.e.}
Let $x\in A_{P}^{\ZZ}$ such that $x_{0}=(\uno, \cherry)$. Then, $x$ is not a mean equicontinuity point.  

\end{lemma}

\begin{proof}
This lemma follows immediately from Lemma \ref{palabra S finito}. 
\end{proof}

Notice that for all $\e>0$, any $y\in B_{\e}(x)$, where $x_{0}=(\uno, \cherry)$, is not a mean  equicontinuity point.

\begin{theorem}
	\label{teorema 2}
$(A_{P}^{\ZZ}, T_{P})$ is neither mean sensitive and nor almost mean equicontinuous.
\end{theorem}

\begin{proof}
Let us show that $(A_{P}^{\ZZ}, T_{P})$ is not mean sensitive, this is, for every $\e>0$ there exists a open set $U\subset A_{P}^{\ZZ}$ such that for every $x,y\in U$ we have that 
$$\limsup_{n\rightarrow \infty} \frac{\sum_{i=0}^{n}d(T_{P}^{i}x,T_{P}^{i}y)}{n+1}< \e .$$
From the Proposition \ref{punto-m-e-lemma}, we have that the element 
$$x:=\cdots(\uno,\cero) ( \cero, \cero)^{2^{1}} (\uno,\cero) ( \cero, \cero)^{2^{0}}.(\uno,\cero) ( \cero, \cero)^{2^{0}} (\uno,\cero) ( \cero, \cero)^{2^{1}} \cdots$$ 
is a mean equicontinuity point. From Proposition \ref{pro sigma}, for every $\e>0$, there exists $\d>0$ such that  for all $y,z \in B_{\d}(x)$ we have that
$$\limsup_{n\rightarrow \infty} \frac{\sum_{i=0}^{n}d(T_{P}^{i}y,T_{P}^{i}z)}{n+1}< \e .$$
Therefore, $(A_{P}^{\ZZ}, T_{P})$ is not mean sensitive.

The fact that $(A_{P}^{\ZZ}, T_{P})$ is not almost mean equicontinuous  follows immediately from Lemma \ref{punto no m.e.}.  
\end{proof}

We finish the paper with a question. 
A minimal TDS is mean equicontinuous if and only if it is not mean sensitive \cite{lituye,weakforms}. Considering Proposition \ref{prop:mindich}, we ask. 

\begin{question}
	Does there exist a minimal subshift $(X,\sigma)$ and a CA $(X,T)$ that is neither mean equicontinuous nor mean sensitive?
\end{question}

\bibliographystyle{plain}
\bibliography{ref}

\begin{thebibliography}{10}

\bibitem{akinauslander}
Ethan Akin, Joseph Auslander, and Kenneth Berg.
\newblock When is a transitive map chaotic.
\newblock {\em Convergence in Ergodic Theory and Probability, Walter de Gruyter
  \& Co}, pages 25--40, 1996.

\bibitem{auslander1980}
Joseph Auslander and James~A. Yorke.
\newblock Interval maps, factors of maps, and chaos.
\newblock {\em Tohoku Matematics Journal}, 32(2):177--188, 1980.

\bibitem{blanchard1997cellular}
Fran{\c{c}}ois Blanchard, Enrico Formenti, and Petr Kurka.
\newblock Cellular automata in the {C}antor, {B}esicovitch, and {W}eyl
  topological spaces.
\newblock {\em Complex Systems}, 11(2):107--124, 1997.

\bibitem{blanchardtisseur}
Fran{\c{c}}ois Blanchard and Pierre Tisseur.
\newblock Some properties of cellular automata with equicontinuity points.
\newblock {\em Annales de l'Institut Henri Poincare (B) Probability and
  Statistics}, 36(5):569--582, 2000.

\bibitem{downarowiczglasner}
Tomasz Downarowicz and Eli Glasner.
\newblock Isomorphic extensions and applications.
\newblock {\em Topological Methods in Nonlinear Analysis}, 48(1):321--338,
  2016.

\bibitem{fomin}
Sergei Fomin.
\newblock On dynamical systems with pure point spectrum.
\newblock {\em Dokl. Akad.Nauk SSSR (in Russian)}, 77(4):29--32, 1951.

\bibitem{fuhrmann2018structure}
Gabriel Fuhrmann, Maik Gr{\"o}ger, and Daniel Lenz.
\newblock The structure of mean equicontinuous group actions.
\newblock {\em arXiv preprint arXiv:1812.10219}, 2018.

\bibitem{garcia2016limit}
Felipe Garc{\'\i}a-Ramos.
\newblock Limit behaviour of $\mu$-equicontinuous cellular automata.
\newblock {\em Theoretical Computer Science}, 623:2--14, 2016.

\bibitem{weakforms}
Felipe Garc{\'\i}a-Ramos.
\newblock Weak forms of topological and measure-theoretical equicontinuity:
  relationships with discrete spectrum and sequence entropy.
\newblock {\em Ergodic Theory and Dynamical Systems}, 37(4):1211--1237, 2017.

\bibitem{garciajagerye}
Felipe Garc{\'\i}a-Ramos, Tobias J{\"a}ger, and Xiangdong Ye.
\newblock Mean equicontinuity, almost automorphy and regularity.
\newblock {\em Israel Journal of Mathematics, in press}.

\bibitem{garcia2019dynamical}
Felipe Garc{\'\i}a-Ramos, Jie Li, and Ruifeng Zhang.
\newblock When is a dynamical system mean sensitive?
\newblock {\em Ergodic Theory and Dynamical Systems}, 39(6):1608--1636, 2019.

\bibitem{gilman1987classes}
Robert~H. Gilman.
\newblock Classes of linear automata.
\newblock {\em Ergodic Theory and Dynamical Systems}, 7(1):105--118, 1987.

\bibitem{guckenheimer}
John Guckenheimer.
\newblock Sensitive dependence to initial conditions for one dimensional maps.
\newblock {\em Communications in Mathematical Physics}, 70(2):133--160, 1979.

\bibitem{huang2018bounded}
Wen Huang, Jian Li, Jean-Paul Thouvenot, Leiye Xu, and Xiangdong Ye.
\newblock Bounded complexity, mean equicontinuity and discrete spectrum.
\newblock {\em Ergodic Theory and Dynamical Systems, in press}.

\bibitem{kurka1997languages}
Petr Kurka.
\newblock Languages, equicontinuity and attractors in cellular automata.
\newblock {\em Ergodic Theory and Dynamical Systems}, 17(2):417--433, 1997.

\bibitem{kuurka2003topological}
Petr Kurka.
\newblock {\em Topological and symbolic dynamics}, volume Cours {S}pecials 11.
\newblock SMF, 2003.

\bibitem{lituye}
Jian Li, Siming Tu, and Xiangdong Ye.
\newblock Mean equicontinuity and mean sensitivity.
\newblock {\em Ergodic Theory and Dynamical Systems}, 35(8):2587--2612, 2015.

\bibitem{lisurveymean}
Jie Li, Xiangdong Ye, and Tao Yu.
\newblock Mean equicontnuity, bounded complexity and applications.
\newblock {\em Discrete and Continuous Dynamical Systems-A, 25 annularity of
  DCDS, in press}.

\bibitem{martin2007damage}
Bruno Martin.
\newblock Damage spreading and $\mu$-sensitivity on cellular automata.
\newblock {\em Ergodic Theory and Dynamical Systems}, 27(2):545--565, 2007.

\bibitem{oxtoby}
John~C. Oxtoby.
\newblock Ergodic sets.
\newblock {\em Bulletin of the American Mathematical Society}, 58(2):116--136,
  1952.

\bibitem{sablik2008directional}
Mathieu Sablik.
\newblock Directional dynamics for cellular automata: {A} sensitivity to
  initial condition approach.
\newblock {\em Theoretical Computer Science}, 400(1-3):1--18, 2008.

\end{thebibliography}

\medskip

\begin{itemize}
	\item \emph{L. de los Santos Ba\~nos, Instituto de F\'isica, Universidad Aut\'onoma de San Luis Potos\'i,  luguis.sb.25@gmail.com}
	\item \emph{F. Garc\'ia-Ramos, CONACyT \& Instituto de F\'isica, Universidad Aut\'onoma de San Luis Potos\'i, fgramos@conacyt.mx}
\end{itemize}	
\end{document}